\documentclass[a4paper, oneside, 10pt]{article} \usepackage[english]{babel}
\usepackage{latexsym,amsfonts,amsmath,enumerate,graphics,enumerate,amsthm}
\usepackage[pagewise]{lineno}
\title{\textbf{Pure point diffraction and entropy beyond the Euclidean space} } \author{T.~Hauser}

\let\epsilon=\varepsilon

\newenvironment{dedication}
        {
        \begin{quotation}\begin{center}\begin{em}}
        {\par\end{em}\end{center}\end{quotation}}

\theoremstyle{definition}
\newtheorem{definition}{Definition}[section]

\theoremstyle{default}

\newtheorem{theorem}[definition]{Theorem}
\newtheorem*{theorem*}{}
\newtheorem{proposition}[definition]{Proposition}
\newtheorem{lemma}[definition]{Lemma}

\newtheorem{exa}[definition]{Example}

\newtheorem*{exa*}{Example}
\newtheorem{remark}[definition]{Remark}
\newtheorem*{acknowledgement}{Acknowledgement}

\let\epsilon=\varepsilon
\let\phi=\varphi
\let\theta=\vartheta

\begin{document}
\maketitle

\begin{dedication}
{Dedicated to the memory of Uwe Grimm} 
\end{dedication}

\begin{abstract}
	For Euclidean pure point diffractive Delone sets of finite local complexity and with uniform patch frequencies it is well known that the patch counting entropy computed along the closed centred balls is zero.
	We consider such sets in the setting of $\sigma$-compact locally compact Abelian groups and show that the topological entropy of the associated Delone dynamical system is zero.
	For this we provide a suitable version of the variational principle. 
	 We furthermore construct counterexamples, which show that the patch counting entropy of such sets can be non-zero in this context. Other counterexamples will show that the patch counting entropy of such a set can not be computed along a limit and even be infinite in this setting. 	
 \end{abstract}



\section{Introduction}

	Aperiodic order is an intriguing field with various motivations. It can be seen as the abstraction of the study of real world solids which are not crystals but have a remarkable long range order. 
	To be more precise in this case the long range order can be observed by studying diffraction experiments and in the observation of a diffraction picture that consists of bright spots \cite{baake2013aperiodic}. The discovery of real world materials that have this form of longe range order, while not being crystals \cite{shechtman1984metallic, baake2013aperiodic} was a revolutionary step in the field of crystallography and even though the field of aperiodic order existed before this discovery showed its importance and relevance for applications in physics. 
	
	Mathematically one considers so called \emph{Delone sets}, i.e.\ subsets $\omega\subseteq \mathbb{R}^d$ for which there exists a compact set $K$ such that $K+\omega=\mathbb{R}^d$ and an open neighbourhood $V$ of $0$ such that $\{V+g;\, g\in \omega\}$ is a disjoint family. The mathematical notion of a crystal is a Delone subgroup and also called a \emph{uniform lattice}. 
	See Subsection \ref{PPDref} for a definition of pure point diffraction. 
	
	Naturally one wants to understand in which sense pure point diffractive materials are ordered. 	
	A possibility to quantify the order of a Delone subset $\omega\subseteq \mathbb{R}^d$ one often uses the \emph{patch counting entropy} \cite{lagarias1999geometric, lagarias2003repetitive, baake2007pure,baake2009kinematic,baake2011kinematic,baake2012comment, HuckandRichard, grimm2015aperiodic}. 
	For $A\subseteq \mathbb{R}^d$ we define 
	$\operatorname{Pat}_\omega(A):=\{(\omega-g)\cap A;\, g\in \omega\}$ and define the patch counting entropy by 
	\[\operatorname{h}_{pc}(\omega):=\limsup_{n\to \infty}\frac{\log|\operatorname{Pat}_\omega(B_n)|}{\theta(B_n)}.\]
	Here $B_n$ denotes the closed centred ball around $0$, $\theta$ is the Lebesgue measure and $|\cdot|$ the cardinality. 	
	In \cite{baake2007pure} it is shown that pure point diffractive subsets of $\mathbb{R}^d$ have $0$ patch counting entropy, which gives strong evidence to our intuition that pure point diffractive materials are somehow ordered. Note that in this article (and in \cite{baake2007pure}) a pure point diffractive Delone set is implicitly assumed to be a Delone set of finite local complexity with uniform patch frequencies. See Section \ref{sec:prelims} below for further details on these notions.

	Another application (and possibly the foundation) of several methods in the field of aperiodic order can be seen in harmonic analysis and in particular the study of $p$-adic numbers. 
	In Y. Meyers famous monograph 'Algebraic numbers and harmonic analysis' \cite{meyer1972algebraic} the method of a 'cut and project scheme' was developed in order to construct a suitable generalization of uniform lattices in locally compact Abelian groups, the so called regular model sets. 
	This method is of particular interest as the additive group of $p$-adic numbers does not contain uniform lattices, but does contain regular model sets. 
	
	On this footing Aperiodic order is nowadays usually studied for subsets of (non-compact) $\sigma$-compact locally compact Abelian groups (LCA$_\sigma$ groups) as for example in \cite{baake2004dynamical, baake2004weighted, strungaru2011positive, HuckandRichard, richard2017pure, lenz2020pure, hauser2020Anote}. 
	In particular, the concepts of pure point diffractiveness and patch counting entropy are defined and studied for subsets of such groups. However, it remained open, whether pure point diffractiveness is sensitive to patch counting entropy and in particular, 
	whether pure point diffraction implies $0$ patch counting entropy in this context. 
	
	Considering $\mathbb{Z}$ equipped with the discrete metric ($d(g,g')=1$ for distinct $g,g'\in \mathbb{Z}$) one observes that in general one cannot use the closed centred balls for averaging in this context. It is standard to use the concept of van Hove sequences instead, which we will introduce below. 
	However, care has to be taken, since the concept of patch counting entropy depends on the choice of a van Hove sequence \cite{hauser2020Anote}. We thus focus in the introduction on LCA$_\sigma$ groups that are equipped with a metric such that the closed centred balls form a van Hove sequence, like for example the additive groups of $p$-adic numbers $\mathbb{Q}_p$ or the Euclidean space $\mathbb{R}^d$. With this assumption we define patch counting entropy as above considering the Haar measure instead of the Lebesgue measure. 
	
	As pointed out by M.\ Baake and U.\ Grimm in \cite{baake2009kinematic,baake2011kinematic,baake2012comment} an absolutely continuous diffraction measure can be insensitive to patch counting entropy, if one considers Delone sets of finite local complexity in the Euclidean space. The result of \cite{baake2007pure} can be interpreted as follows. Pure point diffractiveness is sensitive to the patch counting entropy in the context of Delone sets of finite local complexity and with uniform patch frequencies in the Euclidean space. We will show that the property of being  pure point diffractive is insensitive to the patch counting entropy in the context of Delone sets of finite local complexity and with uniform patch frequencies in the LCA$_\sigma$ group $\mathbb{Q}_2$. 
	The following is a slight simplification of Theorem \ref{the:counterexamples} below.\\

\emph{For each $h\in [0,\infty]$ there exists a pure point diffractive subset (of finite local complexity and with uniform patch frequencies) of the additive group of dyadic numbers $(\mathbb{Q}_2,+)$ that has patch counting entropy $h$.\\}
	
	In \cite[Theorem 2.3]{lagarias1999geometric} it is shown that any finite local complexity Delone subset of $\mathbb{R}^d$ has a finite patch counting entropy. Our result also shows that a similar result does not hold for LCA$_\sigma$ groups, even if the group is metrizable and the subset is considered to be pure point diffractive.

%
%
%
%
	
	
	Analyzing the techniques developed in \cite{baake2007pure} another notion of entropy comes into play. From a pure point diffractive subset $\omega$ a dynamical system is constructed by considering the closure $X_\omega$ of the translation orbit $\{\omega+g;\, g\in G\}$ with respect to an appropriate topology on the set of all closed subsets of the group. Acting on $X_\omega$ with $G$ by translation, i.e.\ considering the action of $(g,\xi)\mapsto \xi+g$ one can use the classical notion of topological entropy of this action \cite{adler1965topological, hauser2020relative, hauser2022entropy}. We will refer to this notion as the \emph{topological entropy of $\omega$}. In contrast to Theorem \ref{the:counterexamples} we have the following result (see Theorem \ref{the:topentropy0} below). Note that we do not need to impose any constraints on the metrizability of the group.\\


\emph{
If $G$ is a $\sigma$-compact locally compact Abelian group then any pure point diffractive subset of finite local complexity  and with uniform patch frequencies has topological entropy $0$.\\
}

	In order to achieve this result we essentially follow the arguments of \cite{baake2007pure}. Our main contribution here is the generalization of the used variational principle to the context of actions of LCA$_\sigma$ groups that do not contain uniform lattices, such as $\mathbb{Q}_p$. Note that it remains open, whether the variational principle holds without the assumption of $\sigma$-compactness. 

	Theorem \ref{the:topentropy0} demonstrates, that the phenomenon observed in Theorem \ref{the:counterexamples} is a result of the difference of the concepts of patch counting and topological entropy for pure point diffractive subsets of $\mathbb{Q}_2$. 
	It is thus natural to ask, how the notion of patch counting entropy should be modified such that this new notion equals the topological entropy for all LCA$_\sigma$ groups. Alternatively, one can ask, how to reformulate the definition of topological entropy of the associated action on $X_\omega$ in order to give an intrinsic geometric definition that does not need the construction of $X_\omega$. 
	This question has already been answered in \cite{hauser2020Anote} for certain LCA$_\sigma$ groups and using the results of \cite{hauser2022entropy} it is straightforward to generalize the results to all LCA$_\sigma$ groups. We finish this introduction with a recap on this approach for the convenience of the reader. 
	
	Let $G$ be a LCA$_\sigma$ group and consider a Delone subset $\omega\subseteq G$. 
	Let $A\subseteq G$ be a compact subset of $G$ and $V$ be an open neighbourhood of $0$. We say that $F\subseteq \omega$ is an \emph{$A$-patch representation at scale $V$ for $\omega$}, whenever for any $g\in \omega$ there is $f\in F$ such that 
	$(\omega-f)\cap A\subseteq (\omega-g)+V$ and $(\omega-g)\cap A\subseteq (\omega-f)+V$. As presented in \cite[Remark 3.8]{hauser2020Anote} there always exists a finite $A$-patch representation at scale $V$ for $\omega$. We define $\operatorname{pat}_\omega(A,V)$ as the minimal cardinality of an $A$-patch representation at scale $V$ for $\omega$. 
	Using \cite[Theorem 4.8]{hauser2022entropy} in \cite{hauser2020Anote} instead of \cite[Lemma 2.5]{hauser2020Anote} we know that \cite[Theorem 3.10]{hauser2020Anote} holds for all LCA$_{\sigma}$ groups and observe the following. 
	\begin{proposition}
		If $G$ is a LCA$_\sigma$ group, then the topological entropy of a Delone subset $\omega$ is given by 
		\[\sup_{V}\limsup_{n\to \infty} \frac{\log \operatorname{pat}_\omega(A_n,V)}{\theta(A_n)},\]
		where $(A_n)_{n\in \mathbb{N}}$ is a van Hove sequence in $G$ and where the supremum is taken over all open neighbourhoods of $0$. 		
	\end{proposition}

\section{Preliminaries}
\label{sec:prelims}

	We denote $\mathbb{N}$ for the set of natural numbers without $0$ and $\mathbb{N}_0:=\mathbb{N}\cup \{0\}$. We denote $\mathbb{Z}[1/2]=\{m/2^n;\, m\in \mathbb{Z}, n\in \mathbb{N}\}$ for the set of \emph{dyadic rational numbers}. 
	In a topological space we denote 
	$\partial A$ for the topological boundary and $\overline{A}$ for the closure of a subset $A$. 

\subsection{Locally compact Abelian groups}

An Abelian group $G$ equipped with a locally compact Hausdorff topology such that the addition $G\times G\to G\colon (g,g')\mapsto g+g'$ and the inverse $G\to G\colon g\mapsto -g$ are continuous is called a \emph{locally compact Abelian group} (\emph{LCA group}). We denote the neutral element of $G$ by $0$. For subsets $A,B\subseteq G$ and $g\in G$ we denote the \emph{Minkowski sum} $A+B:=\{a+b;\, a\in A, b\in B\}$ and $A+g:=\{a+g;\, a\in A\}$. 
	We write $A\oplus B$ for $A+B$, whenever each element of $A+B$ has a unique representation $a+b$ with $a\in A$ and $b\in B$. 
	A LCA group $G$ is called \emph{$\sigma$-compact}, whenever there exists a countable family of compact subsets of $G$ that covers $G$. 
	A LCA group $G$ is called \emph{compactly generated}, whenever there exists a compact subset $K$, such that each element of $G$ can be written as a finite sum of elements from $K$. 
	
\subsubsection{Haar measure}
	
	If $G$ is a LCA group, then there exists a \emph{Haar measure} $\theta$ on $G$, i.e.\  there exists a (non-trivial positive) regular Borel measure $\theta$ that satisfies $\theta(A)=\theta(A+g)$ for all Borel sets $A\subseteq G$ and all $g\in G$. For a Haar measure $\theta$ there holds $\theta(U)>0$ for all non-empty open subsets $U\subseteq G$ and for subsets $A\subseteq G$ there holds $\theta(A)<\infty$, whenever $A$ is precompact. A Haar measure is unique up to scaling, i.e.\ if $\theta$ and $\nu$ are Haar measures on $G$, then there is $c>0$ such that $\theta(A)=c\nu(A)$ for all Borel measurable sets $A\subseteq G$. 
 If nothing else is mentioned, we denote a Haar measure of a LCA group $G$ by $\theta$. Note that below we will pay special interest to $\sigma$-compact groups for which the Haar measure is $\sigma$-finite. 
	For details on these notions and further reference see \cite{Folland,deitmar2014principles}. 
	
\subsubsection{Van Hove sequences}
	Consider a $\sigma$-compact LCA group $G$. 
	For $A,K\subseteq G$ we denote $\partial_K A:=(K+A)\cap\overline{K+(G\setminus A)}$. 
	A sequence $(A_n)_{n\in \mathbb{N}}$ of compact subsets is called \emph{van Hove}, whenever \[\lim_{n\to \infty}\theta(\partial_K A_n)/\theta(A_n)=0\]
	 for each compact subset $K\subseteq G$. 
	
	
	A subset $A\subseteq G$ is called a \emph{tile}, whenever there exists a countable subset $\omega$ of $G$ such that $G=A+\omega$ and such that there holds $\theta((A+g)\Delta (A+h))=0$ for all distinct $g,h\in \omega$. Note that whenever $H$ is a subgroup of $G$, then $A\subseteq H$ is a tile in $H$, if and only if it is a tile in $G$. 
	 A van Hove sequence $(A_n)_{n\in \mathbb{N}}$ is said to be a \emph{van Hove sequence of symmetric tiles}, whenever each $A_n$ is a tile and \emph{symmetric} ($A_n=\{-a;\, a\in A_n\}$). Each compactly generated LCA group $G$ is isomorphic (as a topological group) to an LCA group of the form $\mathbb{R}^a\times \mathbb{Z}^b\times C$ with $a,b\in \mathbb{N}_0$ and a compact Abelian group $C$ \cite[Theorem 4.2.2]{deitmar2014principles} and the latter contains the van Hove sequence of symmetric tiles $([-n,n]^a\times\{-n,\ldots,n\}^b\times C)_{n\in \mathbb{N}}$.  Thus each compactly generated LCA group contains a van Hove sequence of symmetric tiles. More generally the following is shown in \cite[Page 3]{Schlottmann}. 	

\begin{lemma}\label{lem:weisslemmatilingvanhoveexists}
		Any $\sigma$-compact LCA group contains a van Hove sequence of symmetric tiles. 
	\end{lemma}

\subsection{Delone sets}

	A subset $\omega$ of a $\sigma$-compact LCA group $G$ is called \emph{relatively dense}, whenever there exists a compact subset $K\subseteq G$ such that $K+\omega=G$. For a compact neighbourhood of $0$ we furthermore say that $\omega$ is $V$-discrete whenever $\{V+g;\, g \in \omega\}$ is a disjoint family. $\omega$ is called \emph{uniformly discrete}, whenever it is $V$-discrete for some pre-compact neighbourhood of $0$. A Delone set is a uniformly discrete and relatively dense set. 
	
\subsubsection{Finite local complexity and uniform patch frequencies}
	Consider a Delone set $\omega\subseteq G$. For a compact subset $A\subseteq G$ we denote $\operatorname{Pat}_\omega(A):=\{(\omega-g)\cap A;\, g\in \omega\}$. The elements of  $\operatorname{Pat}_\omega(A)$ are called the \emph{$A$-patches} of $\omega$. A \emph{patch of $\omega$} is a set which is an $A$-patch with respect to some compact subset $A\subseteq G$.  
	A Delone set $\omega$ is said to be of \emph{finite local complexity} (\emph{FLC}), whenever  $\operatorname{Pat}_\omega(A)$ is finite for all compact subsets $A\subseteq G$.	
	Given a van Hove sequence $(A_n)_{n\in \mathbb{N}}$ a non-empty $A$-patch $P$ of $\omega$ is said to have a \emph{frequency along $(A_n)_{n\in \mathbb{N}}$}, whenever the limit
\[\lim_{n\to \infty}\frac{|\{g\in A_n;\, (\omega-g)\cap A=P\}|}{\theta(A_n)}\]
exists. $P$ is said to have a \emph{uniform frequency}, whenever it has a frequency along any van Hove sequence $(A_n)_{n\in \mathbb{N}}$. Note that the corresponding limit is then automatically independent of the choice of a van Hove sequence. 

\begin{remark}\label{rem:frequencyGeplapper}
		Let $\omega$ be a FLC Delone set in $G$, $A\subseteq G$ be a compact subset and $P\in \operatorname{Pat}_\omega(A)$. 
		For any van Hove sequence $(A_n)_{n\in \mathbb{N}}$ the following statements are equivalent. 
	\begin{itemize}
	\item[(i)] $P$ has a uniform frequency. 
	\item[(ii)] For all sequences $(s_n)_{n\in \mathbb{N}}$ in $G$ the patch $P$ has a uniform frequency along $(A_n+s_n)_{n\in \mathbb{N}}$. 
	\end{itemize}
	Indeed, it is straight forward to observe from \cite[Corollary 1.12]{lenz2020pure} that for uniformly discrete subsets $u\subseteq G$ the convergence of $|u\cap B_n|/\theta(B_n)$ for all van Hove sequences $B_n$ is equivalent to the convergence of $|u\cap (B_n+s_n)|/\theta(B_n)$ for all sequences $(s_n)_{n\in \mathbb{N}}$ in $G$. Considering the uniformly discrete set $u_{P}:=\{g\in \omega;\, (\omega-g)\cap A=P\}$ the statement follows. 
\end{remark}

 	A FLC Delone set $\omega$ is said to have \emph{uniform patch frequencies} (UPF, also \emph{uniform cluster frequencies}), whenever every non-empty patch of $\omega$ has a uniform frequency.

\subsubsection{Pure point diffraction}\label{PPDref}

	Let $G$ be a $\sigma$-compact LCA group and $\omega$ be a FLC Delone set in $G$.
	A measure $\gamma$ on $G$ is said to be the \emph{unique autocorrelation} of $\omega$, whenever for each van Hove sequence $(A_n)_{n\in \mathbb{N}}$ the sequence 
	$1/\theta(A_n)\sum_{g,g'\in (A_n\cap \omega)}\delta_{g-g'}$ 
converges to $\gamma$ with respect to the weak*-topology on the set of all complex regular Borel measures on $G$. Here $\delta_g$ denotes the \emph{Dirac measure} on $g\in G$. 
	See \cite[Theorem 3.2 and Theorem 3.4]{Schlottmann} for the following result. 
\begin{proposition}\label{pro:existenceofuniqueautocorrelation}
	Every FLC Delone set in $G$ with UPF has a unique autocorrelation. 
\end{proposition}
	
	If a FLC Delone set $\omega$ possesses a unique autocorrelation $\gamma$, then $\gamma$ is positive definite and hence possesses a well defined Fourier transform $\hat{\gamma}$, a positive measure on the dual group $\hat{G}$ \cite[Chapter 4]{argabright1974fourier}.  In this context $\hat{\gamma}$ is called the \emph{diffraction measure of $\omega$}. We recommend \cite{argabright1974fourier} for further details on Fourier transformation.
	
	A FLC Delone set that possesses a unique autocorrelation is said to be \emph{pure point diffractive}, whenever its diffraction measure is a pure point measure.	
	The following geometric characterization of pure point diffraction can be obtained by combining Proposition \ref{pro:existenceofuniqueautocorrelation} with \cite[Theorem 2.13 and Theorem 2.15]{lenz2020pure}. We denote $A\Delta_V B:=(A\setminus (B+V))\cup (B\setminus (A+V))$ for $A,B,V\subseteq G$. 
	
\begin{proposition}\label{pro:geometric_characterization_ppd}
	Let $\omega$ be a FLC Delone set with UPF and $(A_n)_{n\in \mathbb{N}}$ be a van Hove sequence. The set $\omega$ is pure point diffractive, if and only if for each $\epsilon>0$ and each open neighbourhood $V$ of $0$ the set 
	\[\left\{g\in G;\, \limsup_{n\to \infty}\frac{|A_n\cap (\omega\Delta_V(\omega-g))|}{\theta(A_n)}<\epsilon\right\}\] 
	is relatively dense in $G$. 
\end{proposition}

\subsection{Cut and project schemes and regular model sets}

	Consider LCA$_\sigma$ groups $G$ and $H$ and a uniform lattice $\mathcal{L}$ in $G\times H$. Then $(G,H,\mathcal{L})$ is called a \emph{cut and project scheme}, whenever the restricted projection map $\pi_G|_{\mathcal{L}}$ is injective and the projection $\pi_H(\mathcal{L})$ is dense in $H$. 
	
	\begin{exa}
		$(\mathbb{R}, \mathbb{Q}_2, \{(g,g);\, g\in \mathbb{Z}[1/2]\})$ and $(\mathbb{Q}_2,\mathbb{R}, \{(g,g);\, g\in \mathbb{Z}[1/2]\})$ are cut and project schemes. 
	\end{exa} 
	
	Given a subset $W\subseteq H$ with compact closure and nonempty interior that satisfies $\theta_H(\partial W)=0$ we call $\mathcal{L}(W):=\pi_G(\mathcal{L}\cap (G\times W))$ a \emph{regular model set}. 	
	$W$ is called a \emph{window} in this context. 
	
	\begin{proposition}
	\label{pro:regularmodelsetsppd}
		Any regular model set is a pure point diffractive FLC Delone set with UPF. See \cite{Schlottmann}. 
	\end{proposition}

\subsection{Actions of LCA groups}
	Let $G$ be a LCA group and $X$ a compact Hausdorff space. 
	A map $\pi\colon G\times X\to X$ is called a \emph{(continuous) action of $G$}, whenever $\pi$ is continuous and whenever $\pi(g,\pi(g',x))=\pi(g+ g',x)$ for $x\in X$ and $g,g'\in G$. In this context we abbreviate $\pi^g:=\pi(g,\cdot)$, call $G$ the \emph{acting group} and $X$ the \emph{phase space}. For further details on this notion see \cite{auslander1988minimalflows}.

	\subsubsection{Uniformity}
	For a compact Hausdorff space $X$ we call the set $\mathbb{U}_X$ of all neighbourhoods in $X\times X$ of the diagonal $\Delta_X:=\{(x,x);\, x\in X\}$ the \emph{uniformity of $X$}. Note that we obtain our
definition to be a special case of the general definition from \cite[Theorem 6.22]{kelley1959measures} and \cite[Theorem 32.3]{Munkres}.
	A subfamily $\mathbb{B}\subseteq \mathbb{U}_X$ is called a \emph{base} of $\mathbb{U}_X$, if every $\eta\in \mathbb{U}_X$ contains a member of $\mathbb{B}$. 
	$\eta\in \mathbb{U}_X$ is called \emph{symmetric}, whenever $\eta=\{(y,x);\, (x,y)\in \eta\}$. 	
	For $x,y\in X$ and symmetric $\eta \in \mathbb{U}_X$ we say that $x$ and $y$ are \emph{$\eta$-close}, whenever $(x,y)\in \eta$. A subset $E\subseteq X$ is called $\eta$-separated, whenever any two distinct elements in $E$ are not $\eta$-close.

\subsubsection{Local rubber topology}
   We denote by $\mathcal{A}(G)$, $\mathcal{K}(G)$ and $\mathcal{N}(G)$ the set of all closed subsets of $G$, compact subsets of $G$ and open neighbourhoods of $0$, respectively. 
For $K\in \mathcal{K}(G)$, $V\in \mathcal{N}(G)$ and $\xi,\zeta\in \mathcal{A}(G)$ we denote 
	\[\xi \overset{K,V}{\approx} \zeta,\]
	\label{sym:approx}whenever there is $\xi\cap K\subseteq \zeta + V$ and $\zeta\cap K\subseteq \xi + V$. 
 Furthermore we define
	\[\epsilon(K,V):=\left\{(\xi,\zeta)\in \mathcal{A}(G)^2;\, \xi \overset{K,V}{\approx} \zeta \right\}.\]
	\label{sym:epsilonKV}Then there exists a compact Hausdorff topology, called the \emph{local rubber topology}, such that 
	\[\mathbb{B}_{lr}:=\{\epsilon(K,V);\, (K,V)\in \mathcal{K}(G)\times \mathcal{N}(G)\}\]
	\label{sym:lrbase}is a base for the corresponding uniformity $\mathbb{U}_{\mathcal{A}(G)}$ on $\mathcal{A}(G)$. See \cite[Theorem 3]{baake2004dynamical} for reference. 

\subsubsection{Delone actions}
	For a Delone set $\omega\subseteq G$ we denote 
	$D_\omega:=\{\omega+g;\, g\in G\}$
and $X_\omega$ for the closure of $D_\omega$ in $\mathcal{A}(G)$ with respect to the local rubber topology. Then $X_\omega$ is a compact Hausdorff space and $\pi_\omega\colon G \times X_\omega\to X_\omega\colon (g,\xi)\mapsto \xi+g$ is a dynamical system and called the \emph{Delone dynamical system} or the \emph{Delone action} of $\omega$.  For a proof of the continuity of this action see for example \cite{baake2004dynamical}. 

\subsubsection{Topological entropy of actions}
\label{subsubsec:prelims:topologicalentropy}
	Consider an action $\pi$ of a $\sigma$-compact LCA group $G$ on a compact Hausdorff space $X$. 
	For $\eta\in \mathbb{U}_X$ and a compact subset $A\subseteq G$ we denote 
	\[\eta_A:=\{(x,y)\in X^2;\,\forall g\in A \colon (\pi^g(x),\pi^g(y))\in \eta\}\]
	and obtain from \cite[Lemma 4.2]{hauser2020relative} that there holds $\eta_A\in \mathbb{U}_X$.	
	
	Given a symmetric $\eta\in \mathbb{U}_X$ it is straightforward to show that there exist a finite $\eta$-separated subset of $X$ of maximal cardinality. We denote $\operatorname{sep}_X(\eta)$ for this cardinality and define the \emph{topological entropy of $\pi$} as 
	\[\operatorname{h_{top}}(\pi):=\sup_{\eta}\limsup_{n\to \infty} \frac{\log(\operatorname{sep}_X(\eta_{A_n}))}{\theta(A_n)},\]
	where $(A_n)_{n\in \mathbb{N}}$ is a van Hove sequence and the supremum is taken over all symmetric $\eta\in \mathbb{U}_X$. This notion is independent of the choice of a van Hove sequence as shown in \cite{hauser2020relative}.
	Furthermore, this approach is equivalent to the notions of topological entropy considered (in various degrees of generality) in \cite{Tagi-Zade, hauser2020relative,baake2007pure,hauser2022entropy}. 	

\subsection{Measure theoretical aspects of continuous dynamics}

Let $\pi$ be an action of a $\sigma$-compact LCA group $G$ on a compact Hausdorff space $X$. A regular Borel probability measure $\mu$ on $X$ is called \emph{invariant}, whenever for each Borel subset $A\subseteq X$ there holds $\mu(\{\pi^g(x);\, x\in A\})=\mu(A)$. 
	An action, that allows for exactly one invariant Borel probability measure is called \emph{uniquely ergodic}. For a FLC Delone set $\omega$ the corresponding Delone action $\pi_\omega$ is uniquely ergodic, if and only if $\omega$ has UPF \cite[Corollary 3.3]{Schlottmann}. 

\subsubsection{The Halmos-von Neumann theorem}

	If $\mu$ is an invariant regular Borel probability measure on $X$ we call a non-zero function $f\in L^2(X,\mu)$ an \emph{eigenfunction} of $\pi$, whenever there exists a continuous homomorphism $\chi \colon G\to \mathbb{T}$ such that $f\circ \pi^g=\chi(g) f$ holds in $L^2(X,\mu)$ for all $g\in G$. 
	The action $\pi$ is said to have \emph{pure point dynamical spectrum}, whenever the linear span of all eigenfunctions is dense in $L^2(X,\mu)$. 
	
	The action $\pi$ is called a \emph{rotation}, if $X$ can be equipped with an Abelian group structure $(X,+)$, such that $(X,+)$ becomes a compact topological group and such that there exists a group homomorphism $\rho\colon G \to X$ with $\pi(g,x)=\rho(g)+x$. 
	In this case the unique invariant regular Borel probability measure on $X$ is the normalized Haar measure. 
	
		Two actions $\pi$ and $\phi$ of an LCA group $G$ on compact Hausdorff spaces $X$ and $Y$ equipped with invariant regular Borel probability measures $\mu$ and $\nu$ respectively are said to be \emph{measure theoretically conjugated}, whenever there is a measure algebra isomorphism that commutes with the measure algebra homomorphisms induced by $\pi^g$ for $g\in G$. 
	For details on the notion of measure algebras, measure algebra homomorphisms and measure algebra isomorphism, we recommend \cite[Chapter 2]{walters1982introduction}.
		
	A straightforward generalization of standard arguments, for example presented in the proof of \cite[Theorem 3]{baake2007pure}, yields the following version of the Halmos-von Neumann theorem.		
	
	\begin{theorem}[Halmos- von Neumann] \label{the:Halmos-vanNeumann}
		Any action $\pi$ of a LCA group on a compact Hausdorff space equipped with an invariant ergodic regular Borel probability measure  that has pure point dynamical spectrum is measure theoretically conjugated to a rotation. 
	\end{theorem}
		
\subsubsection{Measure theoretical entropy}

	For a finite (Borel-measurable) partition $\alpha$ of $X$ we denote $H_\mu(\alpha):=-\sum_{A} \mu(A)\log(\mu(A))$, where the sum is taken over all $A\in \alpha$ with $\mu(A)\neq 0$. 
	For a finite set $F\subseteq G$ we furthermore denote $\alpha_F$ for the common refinement $\bigvee_{g\in F}\{\pi^{-g}(A);\, A\in \alpha\}:=\{\bigcap_{g\in F}\pi^{-g}(A_g);\, (A_g)_{g\in F}\in \alpha^F\}$. 	 
	Consider a Delone set $\omega$ and a van Hove sequence $(A_n)_{n\in \mathbb{N}}$ in $G$.
	Following \cite{hauser2022entropy} we define the \emph{measure theoretical entropy of $\mu$ and $\pi$} as
	\[h_\mu(\pi):=
	\sup_{\alpha}\limsup_{n\to \infty}\frac{H_\mu(\alpha_{A_n\cap \omega})}{\theta(A_n)},\]
	where the supremum is taken over all finite partitions $\alpha$ of $X$. It is presented in \cite{hauser2022entropy} that this notion does not depend on the choice of the Delone set and the van Hove sequence and that it naturally generalizes the classical notions of entropy like for example considered in \cite{kolmogorov1958new, misiurewicz1976short, stepin1980variational, ollagnier1982variational, Tagi-Zade, baake2007pure}. 

	From \cite[Lemma 3.2]{hauser2020relative} and \cite[Theorem 6.25]{hauser2022entropy} we observe that $h_{top}(\pi)$ is the Ornstein-Weiss topological pressure $p_0^{(OW)}(\pi)$ introduced in \cite{hauser2022entropy}. Note that the measure theoretical entropy introduced in this article is denoted $\operatorname{E}_\mu^{(OW)}(\pi)$ in \cite{hauser2022entropy}. 
	From \cite[Theorem 6.30]{hauser2022entropy} we thus know the following version of Goodwyn's theorem. 
	
\begin{proposition}\label{pro:GoodwynPrelims}
 	For any invariant regular Borel probability measure $\mu$ on $X$ there holds $h_\mu(\pi)\leq h_{top}(\pi)$. 
\end{proposition}

\subsection{The $p$-adic numbers}
\label{sec:prelims_padicnumbers}

	For a field $\mathbb{K}$ a mapping $|\cdot|\colon \mathbb{K}\to [0,\infty)$ is called an \emph{absolute value}, whenever for all $x,y\in \mathbb{K}$ we have 
	\begin{itemize}
	\item[(i)] $|xy|=|x||y|$.
	\item[(ii)]	$|x+y|\leq |x|+|y|$.
	\item[(iii)] $|x|=0$ if and only if $x=0$.
	\end{itemize}
	On $\mathbb{Q}$ well known absolute values are the standard and the trivial absolute value\footnote{The \emph{trivial absolute value} maps $|x|=0$ for $x=0$ and $|x|=1$ for $x\neq 0$.}.
	For a prime number $p$ we can furthermore write any (non-zero) rational number uniquely in the form $x=p^n a/b$ with $a,b,n\in \mathbb{Z}$ and such that $p$ does not divide $a$ and $b$. We then define $|x|_p:=p^{-n}$ (and $|0|_p:=0$) and obtain the \emph{p-adic absolute value} $|\cdot|_p$ on $\mathbb{Q}$. 
	From \cite[Theorem 3.2.13]{gouvea1997p} we know that for each prime $p$ there exists a field $\mathbb{Q}_p$ (the field of \emph{$p$-adic numbers}) with an absolute value $|\cdot|_p$, such that the following statements are valid.
\begin{itemize}
	\item[(i)] $\mathbb{Q}\subseteq \mathbb{Q}_p$ is a dense subfield and $|\cdot|_p$ extends the $p$-adic absolute value of $\mathbb{Q}$. 
	\item[(ii)] $\mathbb{Q}_p$ is complete with respect to the metric introduced by $|\cdot|_p$. 
\end{itemize}
	Note that one can also show that $\mathbb{Q}_p$ with the mentioned properties is unique up to a unique isomorphism of fields with absolute values. For every $g\in \mathbb{Q}_p$ there exists $n\in \mathbb{N}$ and a unique sequence $(g_i)_{i=-n}^\infty$ in $\{0,\cdots,p-1\}$ such that the following series converges in $\mathbb{Q}_p$ and such that 
	$g=\sum_{i=-n}^\infty g_i p^i$. 
	We recommend \cite{gouvea1997p} for further reference on $p$-adic numbers. 

\section{Pure point diffraction implies zero topological entropy}

In the section we show that FLC Delone sets with UPF have $0$ topological entropy, whenever they are pure point diffractive. In order to do this we will first need to prove a suitable version of the variational principle in our context. After this is done the result will follow by a straightforward generalization of the arguments from \cite{baake2007pure}, which we include for the convenience of the reader. 

\subsection{The variational principle}

	The variational principle is well studied in the context of actions of discrete groups \cite{kolmogorov1958new, misiurewicz1976short, stepin1980variational, ollagnier1982variational}.
	However, there is no proof of the full variational principle in the context of actions of groups that do not contain uniform lattices.
	In \cite{hauser2022entropy} it is shown that topological entropy is an upper bound of measure theoretical entropy. 	
	 Since we are in particular interested in $\mathbb{Q}_2$, a group that contains no non-trivial discrete subgroup, we next present a full proof suitable in the context of aperiodic order. This proof combines some ideas of \cite{misiurewicz1976short} with several new ideas, which rely on the existence of van Hove sequences of symmetric tiles. 
	 
	 We denote $C(X)$ for the set of all continuous maps $f\colon X\to \mathbb{R}$ and identify the set $\mathcal{M}(X)$ of all regular Borel probability measures on $X$ with the set of all continuous, positive and linear functionals on $C(X)$ which map the constant map $(x\mapsto 1)$ to $1$. We equip $\mathcal{M}(X)$ with the respective weak*-topology.  
	We denote $\pi_*^g\mu$ for the push forward of $\mu$, i.e.\ $(\pi_*^g\mu)(f):=\mu(f\circ \pi^g)$.

	\begin{lemma}\label{lem:pointwiseContinuityStaticEntropy}
		Let $\alpha$ be a finite (Borel measurable) partition of $X$. The map $\mathcal{M}(X)\ni \nu \mapsto H_\nu(\alpha)$ is continuous at all $\mu\in \mathcal{M}(X)$ which satisfy $\mu\left(\bigcup_{A\in \alpha}\partial A\right)=0$ for all $A\in \alpha$. 
	\end{lemma}
	\begin{proof}
		Recall that a function $f\colon X\to \mathbb{R}$ is called \emph{upper semicontinuous at $x$}, if $f(x)= \limsup_{y\to x}f(x)$. For more details on this notion see \cite[A.1]{downarowicz2011entropy}. From \cite[A.2.7]{downarowicz2011entropy} we know that $\mathcal{M}\ni \nu\to \nu(\overline{A})$ and $\mathcal{M}\ni \nu\to -\nu(A^\circ)$ are upper semicontinuous for $A\in \alpha$, where $A^\circ$ denotes the topological interior of $A$.
	Let $\mu\in \mathcal{M}(X)$ with $\mu\left(\bigcup_{A\in \alpha}\partial A\right)$. For any $A\in \alpha$ we observe 
	\begin{align*}
	\mu(A)
	=\mu(A^\circ)
	&\leq \liminf_{\nu \to \mu}(A^\circ)
	\leq \liminf_{\nu\to \mu}\nu(A)\\
	&\leq \limsup_{\nu \to \mu}\nu(A)
	\leq \limsup_{\nu \to \mu}\nu(\overline{A})
	\leq \mu(\overline{A})=\mu(A)
	\end{align*}
	and hence $\mu(A)=\lim_{\nu\to \mu}\nu(A)$. We conclude that 
	\[\mathcal{M}(X)\ni \nu \mapsto H_\nu(\alpha)=-\sum_{A\in \alpha}\nu(A)\log(\nu(A))\] is continuous at $\mu$. 
	\end{proof}

	Considering a compact subset $A\subseteq G$ with $\theta(A)>0$ and a regular Borel probability measure $\sigma$ on $X$ we denote $1/\theta(A)\int_{A}\pi^g_*\sigma d\theta(g)$ for the mapping $C(X)\ni f\mapsto 1/\theta(A)\int_{A}\pi_*^g \sigma(f) d\theta(g).$ 
	A straightforward argument shows $1/\theta(A)\int_{A}\pi^g_*\sigma d\theta(g)\in \mathcal{M}(X)$. 
A Borel probability measure $\sigma$ is called \emph{finitely supported}, whenever there exists a finite set $E$ and $p_x\in [0,1]$ with $x\in E$ such that $\sum_{x\in E}p_x=1$ and $\sigma=\sum_{x\in E}p_x\delta_x$.

\begin{lemma}\label{lem:measureintegrationproperties}
	Consider a finitely supported regular Borel probability measure $\sigma$ on $X$ and a compact subset $A\subseteq G$ with $\theta(A)>0$. 
	Denote $\mu:={1}/{\theta(A)}\int_{A}\pi^g_*\sigma d\theta(g)$.
	For any finite partition $\alpha$ of $X$ with $\bigcup_{M\in \alpha}\mu(\partial M)=0$ the map 
	 $A\ni g\mapsto H_{\pi^g_*\sigma}(\alpha)$ 
	 is Lebesgue integrable with respect to the Haar measure and satisfies 
	\[\frac{1}{\theta(A)}\int_A H_{\pi^g_*\sigma}(\alpha)d\theta(g)\leq H_{\mu}(\alpha).\]
\end{lemma}
\begin{proof}
	Consider a finite set $E\subseteq X$ and $p_x\in [0,1]$ with $\sum_{x\in E}p_x=1$ such that $\sigma=\sum_{x\in E}p_x\delta_x$. 
	Rescaling the Haar measure if necessary, we assume without lost of generality that there holds $\theta(A)=1$. 
The map $G\ni g\mapsto H_{\pi^g_* \sigma}(\alpha)$ is constant on the elements of the finite and measurable partition 
$\bigvee_{x\in E}\{\pi(\cdot,x)^{-1}(M);\, M\in \alpha\}$ 
of $G$. Restricted to the compact subset $A$ this map is thus Lebesgue integrable.

	Let $\mathfrak{R}$ be the set of all pairs $(\beta,\mathfrak{g})$ such that $\beta$ is a finite (and measurable) partition of $A$ and such that $\mathfrak{g}\colon \beta\to A$ satisfies $\mathfrak{g}(B)\in B$ for any $B\in \beta$. 
	We order $\mathfrak{R}$ by setting $(\beta,\mathfrak{g})\leq (\beta',\mathfrak{g}')$, whenever $\beta'$ is finer than $\beta$, i.e.\ whenever for each $B'\in \beta'$ there is $B\in \beta$ with $B'\subseteq B$. $\mathfrak{R}$ is easily seen to be a directed partially ordered set and a straightforward argument shows that 
	\[\left(\sum_{B\in \beta}\theta(B)\pi^{\mathfrak{g}(B)}_*\sigma\right)_{(\beta,\mathfrak{g})\in \mathfrak{R}}\]
	converges to $\mu$ with respect to the weak*-topology. 
	 From Lemma \ref{lem:pointwiseContinuityStaticEntropy} we observe that $\mathcal{M}(X)\ni\nu\mapsto H_\nu(\alpha)$ is continuous at $\mu$ and in particular that 
	\[H_{\sum_{B\in \beta}\theta(B)\pi^{\mathfrak{g}(B)}_*\sigma}(\alpha) \overset{(\beta,\mathfrak{g})\in \mathfrak{R}}{\rightarrow} H_{\mu}(\alpha).\]
	For $(\beta,\mathfrak{g})\in \mathfrak{R}$ with $\beta$ finer than the induced partition of $\bigvee_{x\in E}\{\pi(\cdot,x)^{-1}(M);\, M\in \alpha\}$ on $A$ we observe from a well known argument (for example contained in the proof of \cite[Theorem 8.1]{walters1982introduction}) that  
	\[\int_A H_{\pi^g_*\sigma}(\alpha) d\theta(g)
	=\sum_{B\in \beta}\theta(B)H_{\pi^{\mathfrak{g}(B)}_*\sigma}(\alpha)
	\leq H_{\sum_{B\in \beta}\theta(B)\pi^{\mathfrak{g}(B)}_*\sigma}(\alpha)\]
	and the claimed inequality follows. 	
\end{proof}

\begin{theorem}[Variational principle]
	\label{the:VP}
	Let $\pi$ be an action of a $\sigma$-compact LCA group $G$ on a compact Hausdorff space $X$. 
	We have $h_{top}(\pi)=\sup_{\mu}h_\mu(\pi),$
	where the supremum is taken over all invariant regular Borel probability measures $\mu$ on $X$. 
\end{theorem}

\begin{remark}
	The techniques presented in this article can be adapted in order to give a proof for the variational principle in the context of topological pressure, as introduced in \cite{hauser2022entropy}. Since we do not need the variational principle in this generality we will only present a version sufficient for our purposes.
\end{remark}

\begin{remark}
	It remains open to give a proof of the full variational principle, whenever $G$ is a unimodular amenable group, or a non-$\sigma$-compact LCA group. 
\end{remark}

\begin{proof}[Proof of Theorem \ref{the:VP}]
	Recall from Proposition \ref{pro:GoodwynPrelims} that for any invariant regular Borel probability measure $\mu$ on $X$ there holds $h_\mu(\pi)\leq h_{top}(\pi)$. 

Consider a van Hove sequence $\mathcal{A}=(A_n)_{n\in \mathbb{N}}$ of symmetric tiles. In order to show the open half of the variational principle it suffices to show that for any $\eta\in \mathbb{U}_X$ there exists a invariant $\mu\in \mathcal{M}(X)$ such that 
\[h_{sep}^\mathcal{A}(\eta)
	:=\limsup_{n\to \infty}\frac{\log(\operatorname{sep}_X(\eta_{A_n}))}{\theta(A_n)}\leq h_\mu(\pi).\]	
 For $n\in \mathbb{N}$ let us consider a (finite) $(\eta_{A_n})$-separated subset $E_n\subseteq X$ of maximal cardinality $\operatorname{sep}_X(\eta_{A_n})$. 
 We define 
 	$\sigma_n:=1/|E_n|\sum_{x\in E_n}\delta_x$
 and furthermore 
 	\[\mu_n:=\frac{1}{\theta(A_n)}\int_{A_n}\pi^g_*\sigma_n d\theta(g).\]
 Restricting to a subsequence of $(A_n)_{n\in \mathbb{N}}$ if necessary we assume without lost of generality that $h_{sep}^\mathcal{A}(\eta)
	=\lim_{n\to \infty}{\log(\operatorname{sep}_X(\eta_{A_n}))}/{\theta(A_n)}$ and that $(\mu_n)_{n\in \mathbb{N}}$ converges in $\mathcal{M}(X)$ to some $\mu$.
	With a standard Krylov-Bogolyubov argument it follows that $\mu$ is invariant and it remains to show that $\mu$ satisfies 
$h_{sep}^\mathcal{A}(\eta)\leq h_\mu(\pi)$.

	Let $\omega$ be a Delone set in $G$ that is $K$-dense with respect to some compact and symmetric subset $K\subseteq G$. 
	A standard argument involving Frodas theorem and the characterization of the topology of a compact Hausdorff space by a family of continuous pseudometrics allow to choose a finite partition $\alpha$ that satisfies $\left(1/2\mu+\sum_{n\in \mathbb{N}}1/2^{n+1}\mu_n\right)(\partial A)=0$ and $A^2\subseteq \eta_K$ for all $A\in \alpha$. In particular, such a choice implies $\mu(\partial A)=0=\mu_n(\partial A)$ for all $A\in \alpha$ and all $n\in \mathbb{N}$. 	
	We abbreviate $F_n:=(K+A_n)\cap \omega$, $2\odot A_n:=A_n+A_n$ and  $3\odot A_n:=A_n+A_n+A_n$ for $n\in \mathbb{N}$.

	Consider now $n,m\in \mathbb{N}$.
	As $A_m$ is a tile there is $\xi\subseteq G$ such that $A_m +\xi =G$ and such that $\theta((A_m+g)\Delta (A_m+h))=0$ for all distinct $g,h\in \xi$. 
	We denote $\check{C}$ for the set of all $g\in \xi$ wich satisfy $2\odot A_m+g\subseteq A_n$. We write $\hat{C}$ for the set of all $g\in \xi$ for which $2\odot A_m +g$ and $A_n$ intersect. These sets satisfy 
	$\hat{C}\setminus \check{C}\subseteq \partial_{(2\odot A_m)} A_n$ and we observe $A_m+(\hat{C}\setminus \check{C})\subseteq \partial_{(3\odot A_m)} A_n$. 
	
	For $a\in A_m$ it is straightforward to show that 
	$A_n\subseteq A_m+\hat{C}+a\subseteq K+F_m+\hat{C}+a$. Thus 
	$\alpha_{F_m+\hat{C}+a}$ is at scale $\eta_{A_m}\supseteq \eta_{K+F_m+\hat{C}+a}=(\eta_K)_{F_m+\hat{C}+a}$. 
	Since $E_n$ was chosen to be $\eta_{A_n}$-separated we thus obtain that every element of the partition $\alpha_{F_m+\hat{C}+a}$ contains at most one element of $E_n$. Thus 
	\[\log \operatorname{sep}_X(\eta_{A_n})=\log|E_n|
	=-\frac{1}{|E_n|}\sum_{x\in E_n}\log\left(\frac{1}{|E_n|}\right)
	= H_{\sigma_n}(\alpha_{F_m+\hat{C}+a})
	\leq \sum_{g\in \hat{C}} H_{\pi_*^{g+a}\sigma_n}(\alpha_{F_m}).\]
Note that $\left(\sup_{\nu\in \mathcal{M}(X)}H_\nu(\alpha_{F_m})\right)\leq |\alpha_{F_m}|<\infty$ and recall that $\hat{C}\subseteq \xi$. This allows to integrate over all $a\in A_m$ and to compute
	\begin{align*}
	\theta(A_m)\log(\operatorname{sep}_X(\eta_{A_n}))
	&\leq \int_{A_m} \sum_{g\in \hat{C}} H_{\pi_*^{g+a}\sigma_n}(\alpha_{F_m}) d\theta(a)
	\\
	&= \int_{A_m+\hat{C}} H_{\pi_*^g\sigma_n}(\alpha_{F_m})d\theta(g)
	\\
	&\leq \int_{A_m+\check{C}} H_{\pi_*^g\sigma_n}(\alpha_{F_m})d\theta(g)
	+\int_{\partial_{(3\odot A_m)}A_n} H_{\pi_*^g\sigma_n}(\alpha_{F_m})d\theta(g)
	\\
	&\leq \int_{A_n} H_{\pi_*^g\sigma_n}(\alpha_{F_m})d\theta(g)
	+\int_{\partial_{(3\odot A_m)}A_n} H_{\pi_*^g\sigma_n}(\alpha_{F_m})d\theta(g)
	\\
	&\leq \int_{A_n} H_{\pi_*^g\sigma_n}(\alpha_{F_m})d\theta(g)
	+\theta(\partial_{(3\odot A_m)}A_n)|\alpha_{F_m}|. 
	\end{align*}	
	Since
$\sigma_n$ is finitely supported and $\mu(\bigcup_{A\in \alpha_{F_n}}\partial A)=0$ we obtain from Lemma \ref{lem:measureintegrationproperties} that
\begin{align*}
	\frac{\log(\operatorname{sep}_X(\eta_{A_n}))}{\theta(A_n)}
	\leq \frac{H_{\mu_n}(\alpha_{F_m})}{\theta(A_m)}+ \frac{\theta(\partial_{(3\odot A_m)}A_n)|\alpha_{F_m}|}{\theta(A_m)\theta(A_n)}. 
\end{align*} 
	Thus the van Hove property of $(A_n)_{n\in \mathbb{N}}$ yields 
	\[h_{sep}^\mathcal{A}(\eta)
	=\lim_{n\to \infty}\frac{\log(\operatorname{sep}_X(\eta_{A_n}))}{\theta(A_n)}\leq \limsup_{n \to\infty} \frac{H_{\mu_n}(\alpha_{F_m})}{\theta(A_m)}+ 0.\]
	Since $\mu(\partial A)=0$ for $A\in \alpha$ and $\mu_n\to \mu$ in $\mathcal{M}(X)$ we obtain 
	\[h_{sep}^\mathcal{A}(\eta)
	\leq \limsup_{m\to \infty}\frac{H_{\mu}(\alpha_{F_m})}{\theta(A_m)} 
	\leq h_\mu(\pi). \]
	\end{proof}

\subsection{Pure point diffraction and topological entropy}

	With the variational principle in our context at hand we next follow the arguments of \cite{baake2007pure} in order to give a proof for our first main result. We include these arguments for the convenience of the reader. 
	
	\begin{theorem}\label{the:topentropy0}
		Let $G$ be a $\sigma$-compact LCA group. 
	 Any pure point diffractive FLC Delone set $\omega\subseteq G$ with UPF has $0$ topological entropy. 
	 \end{theorem}
	 \begin{proof}
	 	Recall that FLC Delone sets with UPF have a unique invariant (and ergodic) Borel probability measure $\mu$ on $X_\omega$. 
	 	From \cite[Theorem 7]{baake2004dynamical} we know that the Delone action of a pure point diffractive FLC Delone set has pure point dynamical spectrum. 
	 	We thus apply the Halmos-von Neumann theorem (Theorem \ref{the:Halmos-vanNeumann}) to $\pi_\omega$ in order to find a measure theoretically conjugated rotation $\phi$ equipped with an invariant ergodic regular Borel probability measure $\nu$. 
	 It is straightforward to show that $\phi$ has $0$ topological entropy and the variational principle (Theorem \ref{the:VP}) implies that $h_\nu(\phi)=0$. 
	 Since measure theoretical entropy is an invariant of measure theoretical conjugation we observe $h_\mu(\pi_\omega)=0$ and another application of the variational principle yields the result. 
	 \end{proof} 
	 
	 \begin{remark}
	 In \cite{baake2015ergodic} the set of visible lattice points $\omega\subseteq\mathbb{Z}^2$ is discussed. Note that this set is not a relatively dense subset of $\mathbb{Z}^2$ and in particular, not Delone. It is presented that this set has pure point diffraction (without having UPF) and positive patch counting entropy. 
	It is furthermore discussed, that this set has non-zero topological entropy. 
	This shows that pure point diffractivness does not necessarily imply $0$ topological entropy for FLC uniformly discrete sets in $\mathbb{Z}^2$ or $\mathbb{R}^2$. 
	 \end{remark}
	 
%
%
%
%
%
%
%
%
%
%

\section{Pure point diffraction and patch counting entropy}

	In this section we will construct several examples of FLC Delone sets in $\mathbb{Q}_2$ with UPF and pure point diffraction, which have non-zero patch counting entropy. 
	Consider a subset $A\subseteq \mathbb{Q}_2$ and $n\in \mathbb{N}$. 
A subset $F\subseteq \mathbb{Q}\cap A$ is called \emph{well placed in $A$}, whenever $F$ contains exactly one element from each ball of radius $1$ contained in $A$.	
	Recall that we write $B_n$ for the closed centred ball in $\mathbb{Q}_2$. We abbreviate $A_n:=B_{2^n}$, $V_n:=B_{2^{-n}}$, 
	\[
	\xi_{n}^k:=\{m/2^{n};\, m\in \mathbb{N}_0, m\leq 2^{n-k}-1\}
	\]
	for $k,n\in \mathbb{N}_0$ with $k\leq n$. 
	For $k\leq n$ we observe
	$A_n=\xi_n^k \oplus A_k$. 
	We write 
	$\xi^k:=\bigcup_{n\in \mathbb{N}}\xi_n^k$ for $k\in \mathbb{N}_0$. 
	Note that $\xi_n^k\subseteq \xi^k\subseteq \xi^0\subseteq \mathbb{Z}[1/2]$ and $\xi_n^k+\xi_k^l=\xi_n^l$ hold for $l,k,n\in \mathbb{N}_0$ with $l\leq k\leq n$.

	\begin{lemma}\label{lem:ppdsets}
		For $k\in \mathbb{N}$ and a finite set $F\subseteq \mathbb{Z}[1/2]$ the set $\xi^k+F$ is a pure point diffractive FLC Delone set in $\mathbb{Q}_2$ with UPF. 
	\end{lemma}
	\begin{proof}
	Consider the cut and project scheme $(\mathbb{Q}_2,\mathbb{R},\{(g,g);\, g\in \mathbb{Z}[1/2]\})$. The considered set $\xi^k+F$ is a regular model set with respect to the window $[0,2^{-k})+F$ and the statement follows from Proposition \ref{pro:regularmodelsetsppd}. 
	\end{proof}



	In order to show that the Delone sets we construct below have uniform patch frequencies the following notion will become useful.
	For $N\in \mathbb{N}$ a van Hove sequence $(C_n)_{n\in \mathbb{N}}$ in $G$ is said to be \emph{$N$-sorted}, whenever for all $n\in \mathbb{N}$ there holds $C_n=C_n+A_N$, i.e.\ whenever $C_n$ consists of finitely many translates of $A_N$.  

\begin{lemma}\label{lem:upfalongsortetsequences}
	Consider a Delone subset $\omega\subseteq  \mathbb{Q}_2$, 
	$A\subseteq G$ compact and a non-empty $A$-patch $P$ of $\omega$. 
	Then $P$ has a uniform frequency, whenever there exists $N\in \mathbb{N}$ such that $P$ has a frequency along any $N$-sorted van Hove sequence. 
\end{lemma}
\begin{proof}
	Assume that there exists $N\in \mathbb{N}$ such that $P$ has a frequency along any $N$-sorted van Hove sequence. 
	Let $(C_n)_{n\in \mathbb{N}}$ be a $N$-sorted van Hove sequence and note that $(C_n+s_n)_{n\in \mathbb{N}}$ is $N$-sorted for any sequence $(s_n)_{n\in \mathbb{N}}$ in $\mathbb{Q}_2$. Thus $P$ has a uniform frequency along $(C_n+s_n)_{n\in \mathbb{N}}$ for all sequences $(s_n)_{n\in \mathbb{N}}$ and the statement follows from Remark \ref{rem:frequencyGeplapper}.
\end{proof}

%

\begin{lemma}\label{lem:constructing_fancy_sets}
	Let $n,a,d\in \mathbb{N}$ with $n+2\leq a$. 
	For a set $F\subseteq \mathbb{Z}[1/2]\cap A_a$ that is well placed in $A_{a}$ and that contains $0$ there exists a set $E\subseteq \mathbb{Z}[1/2]$ such that
	\begin{itemize}
	\item[(a)] $F\subseteq E\cap A_{a+d}$ and $E$ is well placed in $A_{a+d}$. 
	\item[(b)] For a compact subset $A\subseteq A_n$, an $A$-patch $P$ of $E$, $k\leq a$, $h\in \xi_a^0$ and $h'\in \xi^a_{a+d}$ we have 
	\[
	|\{g\in (A_{k}+h);\, (F-g)\cap A=P\}|
	=|\{g\in (A_{k}+h+h');\, (E-g)\cap A=P \}|,\]
	in particular we have $\operatorname{Pat}_F(A)=\operatorname{Pat}_E(A)$. 
	\item[(c)] $E\Delta_{V_n} (F+\xi^{a}_{a+d})=\emptyset$.
	\item[(d)] $2^{d}\leq |\operatorname{Pat}_{E}(A_{m})|\leq |\operatorname{Pat}_{F}(A_{m})|+2^{d+m}$ for $m\in \{n+1,n+2\}$.
\end{itemize}				
\end{lemma}
\begin{proof}
	Note that $F\subseteq \mathbb{Q}$ is a finite subset. 
	Interpreting $F\subseteq \mathbb{R}$ there exists a natural number $r\geq n$ such that $F$ is contained in the centred $\mathbb{R}$-ball of radius $2^r$. 
	We denote 
	\[W:=\{5\cdot 2^r \cdot j;\, j\in \{0,\ldots, 2^{d}-1\}\}.\]
	and observe	$|W|=2^{d}=|\xi_{a+d}^{a}|$. 	
	We choose an injective mapping
	$\xi_{a+d}^{a}\to W;\, g\mapsto v_g$ 
	that satisfies $v_0=0$. For $g\in \xi_{a+d}^a$ we denote
	\[F_g:=
	(F\cap A_n)
	\cup ((F+v_g)\cap (A_{a}\setminus A_n))
\]	
and define $E:=\bigcup_{g\in \xi_{a+d}^a} g+F_g.$
	From $v_0=0$ we observe that there holds $F=0+F_0\subseteq E$. 
%
%
For $g\in \xi_{a+d}^a$ we observe that $F\cap A_n$ is well placed in $A_n$ and that $(F+v_g)\cap (A_{a}\setminus A_n)=v_g+(F\cap (A_{a}\setminus A_n))$ is well placed in $A_{a}\setminus A_n$.
Thus 
$F_g$
	is well placed in $A_{a}$. 
	Since $\xi_{a+d}^{a}$ contains exactly one element from each ball of radius $2^a$ that is contained in $A_{a+d}$ we observe that $E$ is well placed in $A_{a+d}$.

%
%
	Since $A_n$ is a subgroup of $\mathbb{Q}_2$ that contains $W$ it is straightforward to observe (b). 
	From $v_g\in W\subseteq V_n$ for $g\in  \xi_{a+d}^a$ and the definition of $E$ we establish (c).

%
%
In order to show the first inequality of (d) it is sufficient to consider $m=n+1$. Note that $\xi_{a+b}^{a}\subseteq E$.
We will show that the $A_{n+1}$-patches $(E-g)\cap A_{n+1}$ are distinct for distinct $g\in \xi_{a+d}^a$. The statement then follows from $|\xi_{a+d}^a|=2^{d}$. 
	We use that $F$ is well placed in $A_{a}$ and choose $f\in F\cap (A_{n+1}\setminus A_n)$. 
	Let $g\in \xi_{a+d}^{a}$. 
	Then $f+v_g\in (F+v_g)\cap (A_{n+1}\setminus A_n)\subseteq (E-g)\cap A_{n+1}$. 
	Now recall that $f+v_g\in \mathbb{Q}$. Interpreting $f+v_g$ as elements of $\mathbb{R}$ we observe from our choice of $r$ that $f+v_g$ is contained in the $\mathbb{R}$-ball $B_{2^r}^{\mathbb{R}}(v_g)$ of radius $2^r$ around $v_g$. This shows that $(E-g)\cap A_{n+1}$ contains elements of $B_{2^r}^{\mathbb{R}}(v_g)$ for all $g\in \xi_{a+d}^a$. 
	
	For $g\in \xi_{a+d}^a$ recall that $F_g\subseteq A_a$ and hence $F_g=((g+F_g)-g)\cap A_a\subseteq (E-g)\cap A_a$. Now note that two balls of radius $1$ contained in $A_{a+d}$ are either disjoint or equal. Since $A_{a+d}\ni x\mapsto x-g\in A_{a+d}$ is an isometric homeomorphism it thus establishes a bijection on the balls of radius $1$ contained in $A_{a+d}$. 
	We have already seen that $E$ is well placed in $A_{a+d}$ and observe that also $E-g$ is well placed in $A_{a+d}$. In particular, $(E-g)\cap A_a$ is well placed in $A_a$. Since $F_g$ is also well placed in $A_a$ we have 
	$F_g=(E-g)\cap A_a$. 
	Thus 
	\begin{align*}
	(E-g)\cap A_{n+1}
	\subseteq (E-g)\cap A_a
	=F_g\subseteq F\cup (F+v_g)
	\subseteq  B_{2^r}^{\mathbb{R}}(0)\cup B_{2^r}^{\mathbb{R}}(v_g). 
	\end{align*}
	Since $B_{2^r}^{\mathbb{R}}(v_g)$ are distinct for distinct $g\in \xi_{a+d}^a$ by our choice of $W$ we deduce that $(E-g)\cap A_{n+1}$ are indeed distinct for distinct $g\in \xi_{a+d}^a$.

	In order to observe the second inequalities of (d) note that for $m\in \{n+1,n+2\}$ the group structure of $A_m$ yields that the $A_m$ patches $(E-g)\cap A_m$ with 
	$g\in E\cap (\xi_{a+d}^a+(A_{a}\setminus A_m))$ are all contained in $\operatorname{Pat}_F(A_m)$. Thus 
	$|\operatorname{Pat}_E(A_m)|\leq |\operatorname{Pat}_F(A_m)|+|E\cap (\xi_{a+d}^a+A_{m})|=|\operatorname{Pat}_F(A_m)|+2^{d+m}.
	$
\end{proof}

\begin{theorem}\label{the:counterexamples}
	For each $r,s\in [0,\infty]$ with $s\leq r$ there exists a FLC Delone set $\omega\subseteq \mathbb{Z}[\frac{1}{2}]$ in $\mathbb{Q}_2$ such that 
	\begin{itemize}
		\item[(i)] $\omega$ has UPF and is pure point diffractive. 
		\item[(ii)] $h_{pat}(\omega)=\limsup_{n\to \infty}\frac{\log|\operatorname{Pat}_\omega(B_n)|}{\theta(B_n)}=r$.
		\item[(iii)] $\liminf_{n\to \infty}\frac{\log|\operatorname{Pat}_\omega(B_n)|}{\theta(B_n)}=s$. 
	\end{itemize}
\end{theorem}
\begin{remark}
\begin{itemize}
\item[(i)]
	It is natural to ask, whether the constructed $\omega$ are actually regular model sets with respect to a cut and project scheme. This is not the case, as regular model sets have $0$ patch counting entropy as discussed in \cite{HuckandRichard}. 
\item[(ii)]  Note that by Theorem \ref{the:topentropy0} it follows that the topological entropy of the Delone action $\pi_\omega$ is $0$. Thus $\omega$ is an example of a FLC Delone set for which the topological entropy and the patch counting entropy do not coincide. They coincide in the Euclidean setting \cite{baake2007pure}. 
\item[(iii)] Choosing $r\neq s$ we observe that the limit in the patch counting formula does not exist, even if we consider the average along the centred closed balls. This limit exists in the Euclidean setting \cite{hauser2020Anote}. 
\item[(iv)] Choosing $s>0$ we observe that we cannot change to the limit inferior in the definition of the patch counting entropy in order to get that pure point diffractive Delone sets have 0 patch counting entropy. 
\item[(v)] Choosing $r=\infty$ allows to observe that the patch counting entropy of a FLC Delone set can even be infinite. This is not possible for FLC Delone subsets of the Euclidean space as discussed in \cite{lagarias1999geometric}. 
\end{itemize}

\end{remark}
\begin{proof}
	Note that $\{B_n;\, n\in \mathbb{N}\}=\{A_n;\, n\in \mathbb{N}\}$ and hence we have 
	\[\limsup_{n\to \infty}\frac{\log|\operatorname{Pat}_\omega(B_n)|}{\theta(B_n)}=\limsup_{n\to \infty}\frac{\log|\operatorname{Pat}_\omega(A_n)|}{\theta(A_n)}\] and a similar statement about the limit inferior. We will thus consider only the van Hove sequence $(A_n)_{n\in \mathbb{N}}$ in the following. 

	Any sequence $(f_n)_{n\in \mathbb{N}}$ in $\mathbb{N}$ with $f_n=f_{n-1}+f_{n-2}$ for $n\geq 3$ satisfies $f_n/2^n\to 0$. 
	This allows to choose a sequence $(d_n)_{n\in \mathbb{N}}$ in $\mathbb{N}$ such that $d_n\geq d_{n-1}+d_{n-2}$ and which furthermore satisfies 
	$\limsup_{n\to \infty} d_n/2^n=2r/\log(2)$ and $\liminf_{n\to \infty} d_n/2^n=2s/\log(2)$.  
	W.l.o.g.\ we furthermore assume $d_2\geq 3$ and consider $a_n:=3+\sum_{i=1}^{n-1} d_i$ for $n\geq 1$. Note that $a_1=3$ and that $a_{n+1}=a_n+d_n$ for $n\geq 1$. 
	Clearly $d_2\geq a_1$ and $a_1\geq 1+2$. 
	For $n\geq 2$ we observe inductively $d_{n+1}\geq d_{n}+d_{n-1}
	\geq a_{n-1}+d_{n-1}
	=a_{n}$ and $a_n= a_{n-1}+d_{n-1} \geq n-1 +2 +1 = n+2$. 
	Thus for all $n\in \mathbb{N}$ we have $d_{n+1}\geq a_n\geq n+2$. 
	We set $\omega_1:=\xi_{a_1}^0$ and observe that $\omega_1$ is well placed in $A_{a_1}$ and contains $0$. 
	We apply 
	Lemma \ref{lem:constructing_fancy_sets} inductively in order to obtain a family of subsets of $\mathbb{Z}[1/2]$ such that 
	\begin{itemize}
	\item[(a)] $\omega_n\subseteq \omega_{n+1}$ and $\omega_{n+1}$ is well placed in $A_{a_{n+1}}$. 
	\item[(b)] For a compact subset $A\subseteq A_n$, an $A$-patch $P$ of $\omega_{n+1}$, $k\leq a_n$, $h\in \xi_{a_n}^0$ and $h'\in \xi^{a_n}_{a_{n+1}}$ we have 
	\[
	|\{g\in (A_k+h);\, (\omega_n-g)\cap A=P\}|
	=|\{g\in (A_k+h+h');\, (\omega_{n+1}-g)\cap A=P \}|,\]
	in particular we have $\operatorname{Pat}_{\omega_n}(A)=\operatorname{Pat}_{\omega_{n+1}}(A)$. 
	\item[(c)] $\omega_{n+1}\Delta_{V_n} (\omega_n+\xi^{a_n}_{a_{n+1}})=\emptyset$.
	\item[(d)] $2^{d_n}\leq |\operatorname{Pat}_{\omega_{n+1}}(A_{m})|\leq |\operatorname{Pat}_{\omega_{n}}(A_{m})| +2^{d_n+m}$ for $m\in \{n+1,n+2\}$.
\end{itemize}		
		From (a) and (b) we observe that $\omega:=\bigcup_{n\in \mathbb{N}}\omega_n$ is a FLC Delone set.
	Using (b) and (d) with $m=n+1$ we compute 
	\[
	\frac{\log|\operatorname{Pat}_\omega(A_{n+1})|}{\theta(A_{n+1})}
	=\frac{\log|\operatorname{Pat}_{\omega_{n+1}}(A_{n+1})|}{\theta(A_{n+1})}
	\geq \frac{\log (2^{d_n})}{2^{n+1}}
	=\frac{\log(2)}{2}\frac{d_n}{2^n}.\]
	Furthermore, (b) and (d) imply
	\begin{align*}
		|\operatorname{Pat}_\omega(A_{n+1})|
		&=|\operatorname{Pat}_{\omega_{n+1}}(A_{n+1})|\\
		&\leq |\operatorname{Pat}_{\omega_{n}}(A_{n+1})|+2^{d_{n}+n+1}\\
		&\leq |\operatorname{Pat}_{\omega_{n-1}}(A_{n+1})|+2^{d_{n-1}+n+1}+2^{d_{n}+n+1}\\
		&\leq 2^{a_{n-1}}+2^{d_n+n+2}.\\
		&\leq 2^{d_n+n+3}. 
	\end{align*}
	Thus
	\[\frac{\log|\operatorname{Pat}_\omega(A_{n+1})|}{\theta(A_{n+1})}
	\leq \frac{\log (2^{d_n+n+3})}{2^{n+1}}
	=\frac{\log(2)}{2}\frac{(d_n+n+3)}{2^n}\]
	and we observe that ${\log|\operatorname{Pat}_\omega(A_{n+1})|}/{\theta(A_{n+1})}$ and $(\log(2)d_n)/(2 \cdot 2^n)$ have the same accumulation points and conclude (ii) and (iii).

	We next show that $\omega$ has uniform patch frequencies and consider a compact subset $A\subseteq \mathbb{Q}_2$ and an $A$-patch $P$ of $\omega$. There is $N\in \mathbb{N}$ such that $A\subseteq A_{N}$. Let now $(C_m)_{m\in \mathbb{N}}$ be a $A_{a_N}$-sorted van Hove sequence. By Lemma \ref{lem:upfalongsortetsequences} it will be sufficient to show that 
	$P$ has a frequency along $(C_m)_{m\in \mathbb{N}}$. 
	For this we consider $m\in \mathbb{N}$ with $m\geq a_N$. There exist $M\in \mathbb{N}$ and a finite set $F\subseteq \xi_{a_M}^{a_N}$ such that $C_m=F\oplus A_{a_N}\subseteq A_{a_M}$.
	Applying (b) inductively we observe that for $h\in F\subseteq \xi_{a_M}^{a_N}=\xi_{a_M}^{a_{M-1}}\oplus \dots \oplus \xi_{a_{N+1}}^{a_N}$ we have 
	\[|\{g\in A_{a_N};\, (\omega_N-g)\cap A=P\}|
	=|\{g\in (A_{a_N}+h);\, (\omega_M-g)\cap A=P\}|.\]
	Thus
	\begin{align*}
	|\{g\in C_m;\, (\omega-g)\cap A=P\}|
	&=|\{g\in C_m;\, (\omega_M-g)\cap A=P\}|\\
	&=\sum_{h\in F}|\{g\in (A_{a_N}+h);\, (\omega_M-g)\cap A=P\}|\\
	&=|F| |\{g\in A_{a_N};\, (\omega_N-g)\cap A=P\}|\\
	&=\theta(C_m)/\theta(A_{a_N}) |\{g\in A_{a_N};\, (\omega_N-g)\cap A=P\}|. 
	\end{align*}
	Thus $P$ has the frequency $|\{g\in A_{a_N};\, (\omega_N-g)\cap A=P\}|/\theta(A_{a_N})$ along $(C_m)_{m\in \mathbb{N}}$. 
	
	It remains to show that $\omega$ has pure point diffraction. 
	Note that $\rho^l:=\xi^{a_l}+\omega_l$ has uniform patch frequencies and pure point diffraction as a subset of $\mathbb{Q}_2$ by Lemma \ref{lem:ppdsets}. 
	From Proposition \ref{pro:geometric_characterization_ppd} we thus deduce that for $\epsilon>0$ the set of all $g\in \mathbb{Q}_2$ for which 
	\[\limsup_{n\to \infty} \frac{|(\rho^l\Delta_{V_l} (\rho^l-g))\cap A_n|}{\theta(A_n)}<\epsilon\]
	is relatively dense in $\mathbb{Q}_2$. 	
	Recall that $V_l$ is a subgroup of $\mathbb{Q}_2$ and hence 
	\[\{(\tau,\tau')\in \mathcal{P}(\mathbb{Q}_2)^2;\, \tau\Delta_{V_l}\tau'=\emptyset\}\]
	is an equivalence relation. This observation allows to deduce $\omega \Delta_{V_l} \rho^l=\emptyset$ from (c). Furthermore, $V_l$ being a subgroup yields that for all $g\in \mathbb{Q}_2$ we have 
	$|(\omega\Delta_{V_l} (\omega-g))\cap A_n|=|(\rho^l\Delta_{V_l} (\rho^l-g))\cap A_n|$.
	Thus also the set of all $g\in \mathbb{Q}_2$ with 
	\[\limsup_{n\to \infty} \frac{|(\omega\Delta_{V_l} (\omega-g))\cap A_n|}{\theta(A_n)}<\epsilon\]
	is relatively dense in $\mathbb{Q}_2$. Proposition \ref{pro:geometric_characterization_ppd} implies that $\omega$ has pure point diffraction.
\end{proof}

\begin{acknowledgement}
The author would like to thank Michael Baake for useful suggestions that helped improving this article. Furthermore, the author wants to express his gratitude towards the anonymous referee for his careful reading and various suggestions. In particular, the referee gave an idea that helped simplifying the proof of Lemma \ref{lem:upfalongsortetsequences}. 
\end{acknowledgement}

\footnotesize
\bibliographystyle{alpha}
\bibliography{referencesPPDaEbRd}

 \vspace{10mm} \noindent
\begin{tabular}{l l }
Till Hauser, Max-Planck-Institut für Mathematik, Vivatsgasse 7, Bonn, Germany\\
Email address: hauser.math@mail.de
\end{tabular}

\end{document}